\documentclass[preprint,review,10pt]{amsart}
\usepackage{amsmath,amssymb,amsfonts,xspace,mathrsfs}
\newtheorem{theorem}{Theorem}[section]

\theoremstyle{definition}
\newtheorem{definition}[theorem]{Definition}
\newtheorem{example}[theorem]{Example}

\newtheorem{proposition}[theorem]{Proposition}

\theoremstyle{remark}

\numberwithin{equation}{section}

\begin{document}

\title{$(n,Q)$-ideals and $\phi$-$(n,Q)$-ideals of commutative rings }

\author{Mahdi Anbarloei}
\address{Department of Mathematics, Faculty of Sciences,
Imam Khomeini International University, Qazvin, Iran.
}

\email{m.anbarloei@sci.ikiu.ac.ir}





\begin{abstract}
 In this paper, we introduce and study the notions of  $(n,Q)$-ideals and $\phi$-$(n,Q)$-ideals of commutative rings. 

\end{abstract}
\maketitle

\section{Introduction} 
Throughout this paper, we assume  that all rings are commutative with identity. Let $A$ be a commutative ring and $I$ be an ideal of $A$. We will denote by $\mathcal{I}(A)$, the the set of all ideals of $A$. Let $a_1,\ldots,a_n$ be elements in $A$ where  $n$ is a positive integer. Then,  $a_1\cdots \widehat{a_t} \cdots a_n$ indicates that $a_t$ is omitted from $a_1 \cdots a_n$.

The significance of prime and primary ideals in ring theory is highlighted by the extensive volume of research dedicated to these concepts and their extended forms. The study of generalized prime ideals took a major turn in 2007 upon Ayman Badawi's introduction of 2-absorbing ideals within commutative rings \cite{Badawi}. Later, Ulucak et al. introduced the notion of $n$-absorbing $\delta$-primary ideals as a generalization of 2-absorbing ideals where $\delta: \mathcal{I}(A) \to \mathcal{I}(A)$   is a function satisfying certain properties \cite{Ulucak}. A proper ideal $I$ of $A$ is called an $n$-absorbing $\delta$-primary ideal if whenever, $a_1 \cdots a_{n+1} \in I$ for $a_1, \ldots, a_{n+1} \in A$, then $a_1 \cdots a_n \in I$ or $a_1 \cdots \widehat{a_t} \cdots  a_{n+1} \in \delta(I)$ for some $1 \leq t \leq n$. Let $\phi: \mathcal{I}(A) \to  \mathcal{I}(A) \cup \{\varnothing\}$ be a function. Mostafanasab and Darani presented the idea of $\phi$-$n$-absorbing primary ideals. They said that a proper ideal $I$ of $A$ is   a $\phi$-$n$-absorbing  primary ideal if whenever, $a_1 \cdots a_{n+1} \in I \backslash \phi(I)$ for $a_1, \ldots, a_{n+1} \in A$, then $a_1 \cdots a_n \in I$ or $a_1 \cdots \widehat{a_t} \cdots  a_{n+1} \in \sqrt{I}$ for some $1 \leq t \leq n$. Let $J$ and $N$ be the Jacobson radical and nilradical of $A$, respectively. The notions  of $J$-ideals  and $N$-ideals  were proposed in \cite{Khashan2} and \cite{Tekir}, respectively. To unify the concepts  of prime ideals, primary ideals, $J$-ideals and $n$-ideals, 
Mimouni introduced a new class of ideals called $Q$-ideals  \cite{Mimouni}. Let $I \subseteq Q$ be ideals of $A$. The ideal $I$ is a $Q$-ideal if for every $a,b \in A$, $ab \in I$ and $a \notin Q$ implies that $b \in I$. In \cite{Khalfi}, Khalfi generalized the concetp of $J$-ideals to $\phi$-$(n,I)$-ideals  and in \cite{Anebri},  Anebri et al. extended the notion of $N$-ideals to $\phi$-$(n,N)$-ideals.

 In this paper, our purpose is to introduce and study the notion of $(n,Q)$-ideals as a generalization of $Q$-ideals.  Furthermore, to establish a common framework for $\phi$-$n$-absorbing  ideals, $\phi$-$n$-absorbing primary ideals,  $\phi$-$(n,J)$-ideals, and $\phi$-$(n,N)$-ideals, we put forward the notion of $\phi$-$(n,Q)$-ideals.






\section{$(n,Q)$-ideals}

\begin{definition}
Let $I$ and $ Q$  be proper ideals of $A$. We say that $I$ is  an  $(n,Q)$-ideal of $A$ if whenever $a_1\cdots a_{n+1} \in I  $ for $a_1,\ldots,a_{n+1} \in A$, then $a_1\cdots  a_n \in I$ or there exists $1 \leq t \leq n$ such that $a_1\cdots \widehat{a_t} \cdots a_{n+1} \in Q.$
\end{definition}
\begin{example}
Let $A=\mathbb{Z}[x,y,z]$, $Q=\langle xy,z  \rangle$ and  $I=\langle xy \rangle$. Then, $I$  is a $(2,Q)$-ideal of $A$, but it is not $Q$-ideal. Indeed, $x  y \in I$ and $x,y \notin Q$ but  $y,x \notin I$.
\end{example}
\begin{proposition} \label{1}
Let $I $   be a proper ideal of $A$. If $I$ is an  $(n,Q)$-ideal of $A$ where $Q$ is a proper ideal of $A$, then $I   \subseteq Q$.
\end{proposition}
\begin{proof}
Suppose that $I$ is an $(n,Q)$-ideal for some proper ideal $Q$ of $A$  such that  $I    \nsubseteq Q$. Then, there exists $a \in I$ such that $a \notin Q$. Since $I$ is a  $(n,Q)$-ideal of $A$ and $1 \cdots 1\cdot  a\in I$ , we get $1  \in I$ which is impossible. Hence, $I \subseteq Q$, as needed.
\end{proof}
Recall from \cite{Khalfi} that a proper ideal $I$ of $A$ is an  $(n,J)$-ideal if whenever $a_1\cdots a_{n+1} \in I  $ for $a_1,\ldots,a_{n+1} \in A$, then $a_1\cdots  a_n \in I$ or there exists $1 \leq t \leq n$ such that $a_1\cdots \widehat{a_t} \cdots a_{n+1} \in J.$ 
\begin{theorem}\label{2}
Let $I$ be a proper ideal of $A$. Then, $I$ is an $(n,J)$-ideal of $A$ if and only if $I$ is an $(n,M)$-ideal for each  $M \in \text{Max}(A)$.
\end{theorem}

\begin{proof}
Let $I$ be a $n$-$J$-ideal of $A$ and $M \in \text{Max}(A)$. Suppose that $a_1\cdots a_{n+1} \in I$ for $a_1,\ldots,a_n \in A$ and $a_1\cdots \widehat{a_t} \cdots a_{n+1} \notin M$ for some $1 \leq t \leq n$. It follows that  $a_1\cdots \widehat{a_t} \cdots a_{n+1} \notin \text{J}(A)$. Since $I$ is an $(n,J)$-ideal of $A$, we get $a_1\cdots a_n \in I$. Thus, $I$ is a $(n,M)$-ideal of $A$. Conversely, assume that  $I$ is a $n$-$M$-ideal for each  $M \in \text{Max}(A)$. So, we obtain  $I \subseteq M$ for every $M \in \text{Max}(A)$ by Proposition \ref{1}. This  means $\text{J}(A)$ contains $I$. Now, suppose that $a_1\cdots a_{n+1} \in I$ for $a_1,\ldots,a_n \in A$ and $a_1\cdots \widehat{a_t} \cdots a_{n+1} \notin \text{J}(A)$ for some $1 \leq t \leq n$. Then, we conclude that  $a_1\cdots \widehat{a_t} \cdots a_{n+1} \notin M_0$ for some $M_0 \in \text{Max}(A)$. By the hypothesis, $I$  is an $(n,M_0)$-ideal of $A$. Therefore, we get $a_1\cdots a_n \in I$. Thus,  $I$ is an $(n,J)$-ideal of $A$.
\end{proof}
\begin{proposition} \label{2.1}
Assume that  $A$ is a commutative ring. 
\begin{enumerate}
\item If $I$ is an  $(n,Q)$-ideal   of $A$ for each ideal $Q$   properly containing $I$, then $I$ is an $n$-absorbing primary ideal.
\item  If $I$ is an $(n,Q)$-ideal of $A$ for each ideal $Q$   properly containing $\sqrt{I}$, then $\sqrt{I}$ is an $n$-absorbing  ideal.
\item The following conditions are equivalent:
\begin{enumerate}
\item[(i)] $I$ is an  $n$-absorbing primary ideal.
\item[(ii)] $I$ is an $(n,\sqrt{I})$-ideal.
\item[(iii)] $I$ is an $(n,Q)$-ideal of $A$ for each ideal $Q$ of $A$ containing $\sqrt{I}$.
\end{enumerate}
\end{enumerate}
\begin{proof}
(1) Suppose that $a_1\cdots a_{n+1} \in I$ for $a_1,\ldots, a_{n+1} \in A$ but $a_1 \cdots a_n \notin I$. Set, $Q=I+\langle a_1 \cdots a_n \rangle $. Then, $I$ is an $(n,Q)$-ideal of $A$ as $I \subsetneq Q$. This implies that $a_1 \cdots \widehat{a_t} \cdots a_{n+	1} \in Q$ for some $1 \leq t \leq n$. Then, there exist $x \in I$ and $y \in A$ such that $a_1 \cdots \widehat{a_t} \cdots a_{n+	1}=x+y(a_1 \cdots a_n)$. Hence, \\

$\hspace{0.5cm}(a_1 \cdots \widehat{a_t} \cdots a_{n+	1})^2=x(a_1 \cdots \widehat{a_t} \cdots a_{n+	1})+y(a_1 \cdots a_n)(a_1 \cdots \widehat{a_t} \cdots a_{n+	1})$

$\hspace{3.4cm}=x(a_1 \cdots \widehat{a_t} \cdots a_{n+	1})+y(a_1 \cdots a_{n+1})(a_1 \cdots \widehat{a_t} \cdots a_n)$.

$\hspace{3.4cm} \in I,$\\

which implies $ a_1 \cdots \widehat{a_t} \cdots a_{n+	1} \in \sqrt{I}$. Therefore, $I$ is an $n$-absorbing primary ideal.

(2) Let $I$ be  an $(n,Q)$-ideal of $A$ for each ideal $Q$ of $A$ properly containing $\sqrt{I}$. Suppose  that $a_1 \cdots a_{n+1} \in \sqrt{I}$ and $a_1 \cdots a_n \notin \sqrt{I}$. This means  that $(a_1 \cdots a_{n+1})^k=a_1^k \cdots a_{n+1}^k\in I$ for some $k \in \mathbb{N}$. Put $Q=\sqrt{I}+\langle a_1 \cdots a_n \rangle$. By the hypothesis, $I$ is an $(n,Q)$-ideal as $\sqrt{I} \subsetneq Q$. Since $I$ is  an $(n,Q)$-ideal of $A$, $a_1^k \cdots a_{n+1}^k\in I$ and $a_1^k \cdots a_n^k\notin I$, we obtain $(a_1 \cdots \widehat{a_t} \cdots a_{n+1})^k=a_1^k \cdots \widehat{a_t^k} \cdots a_{n+1}^k\in Q$ for some $1 \leq t \leq n$. There exists $x \in \sqrt{I}$ and $y \in A$ such that $(a_1 \cdots \widehat{a_t} \cdots a_{n+1})^k=x+y(a_1 \cdots a_n)$. Therefore, we have\\

$(a_1 \cdots \widehat{a_t} \cdots a_{n+1})^{2k}=x(a_1 \cdots \widehat{a_t} \cdots a_{n+1})^k+y(a_1 \cdots a_n)(a_1 \cdots \widehat{a_t} \cdots a_{n+1})^k$

$\hspace{3.1cm}=x(a_1 \cdots \widehat{a_t} \cdots a_{n+1})^k+y(a_1 \cdots a_{n+1})(a_1 \cdots \widehat{a_t} \cdots a_n)^k a_{n+1}^{k-1}$

$\hspace{3.1cm} \in \sqrt{I}.$\\
Hence, $a_1 \cdots \widehat{a_t} \cdots a_{n+1} \in\sqrt{I}$. Consequently, $\sqrt{I}$ is an $n$-absorbing  ideal.

(3) (i) $\Longrightarrow$ (ii) and (ii) $\Longrightarrow$ (iii) are clear.

(iii) $\Longrightarrow$ (i)  Assume that $a_1 \cdots a_{n+1} \in I$ for $a_1,\ldots, a_{n+1} \in A$. Since $I$ is an $(n,\sqrt{I})$-ideal by the hypothesis, we have $a_1\cdots a_n \in I$ or $a_1 \cdots \widehat{a_t} \cdots a_{n+1} \in \sqrt{I}$ for some $1 \leq t \leq n$, as desired. 
\end{proof}
\end{proposition}
\begin{theorem} \label{3}
Let $I \subseteq Q$ be ideals of $A$. If $Q$ is an $(n-1)$-absorbing ideal of $A$, then $I$ is an $(n,Q)$-ideal of $A$. 
\end{theorem}
\begin{proof}
Assume that $a_1\cdots a_{n+1} \in I$ for some $a_1,\ldots,a_{n+1} \in A$ but $a_1\cdots a_n \notin I$. Let $Q$ be  an $(n-1)$-absorbing ideal of $A$. If $a_1\cdots a_n \in Q$, then we obtain $a_1\cdots a_{n-1} \in Q$ or $a_1\cdots \widehat{a_t} \cdots a_{n} \in Q$ for some $1 \leq t \leq n-1$. This implies that $a_1\cdots a_{n-1} \widehat{a_n} a_{n+1} \in Q$ or $a_1\cdots \widehat{a_t} \cdots a_{n}a_{n+1} \in Q$ for some $1 \leq t \leq n-1$.
If $a_1\cdots a_n \notin Q$, then we have $(a_1a_2)a_3\cdots  \widehat{a_t}\cdots a_{n+1} \in Q$ for some $1 \leq t \leq n$ as $Q$ is an $(n-1)$-absorbing ideal of $A$ and $(a_1a_2)a_3\cdots   a_{n+1} \in Q$. Thus, $I$ is an $(n,Q)$-ideal of $A$. 
\end{proof}

\begin{proposition} \label{4}
Let $Q$ be an   ideal of $A$. If $I_j$ is an $(n_j,Q)$-ideal of $A$ for each $1 \leq j \leq k$, then $\cap_{j=1}^k I_j$ is an $(n,Q)$-ideal where $n=\Sigma_{j=1}^k n_j$. 
\end{proposition}
\begin{proof}
Let $a_1 \cdots a_{n+1} \in \cap_{j=1}^k I_j$ for $a_1,\ldots,a_{n+1} \in A$ but $a_1\cdots \widehat{a_t} \cdots a_{n+1} \notin Q$ for all $1 \leq t \leq n$.  Since $I_j$ is an $(n_j,Q)$-ideal of $A$ and $a_1 \cdots a_{n+1} \in I_j$ for all $1 \leq j \leq k$, there exist elements  $1 \leq \alpha_1^j, \alpha_2^j,\ldots,\alpha_{n_j}^j \leq n$ with $a_{\alpha_1^j}, a_{\alpha_2^j},\ldots,a_{\alpha_{n_j}^j}  \in I_j$. Let $\alpha_l^r=\alpha_m^s$ for pairs $r,l$ and $s,m$. Then,

\[a_{\alpha_1^1}  a_{\alpha_2^1}  \cdots  a_{\alpha_{n_1}^1}   \cdots a_{\alpha_1^r}  a_{\alpha_2^r}   \cdots   a_{\alpha_l^r}   \cdots   a_{\alpha_ {n_r}^r}   \cdots    a_{\alpha_1^s}   a_{\alpha_2^s}\]
\[ \cdots       \widehat{a_{\alpha_m^s}} \cdots  a_{\alpha_{n_s}^s} \cdots    a_{\alpha_1^k}   a_{\alpha_2^k}   \cdots   a_{\alpha_{n_k}^k} \in Q \]
which implies $a_1 \cdots \widehat{a_{\alpha_m^s}} \cdots a_na_{n+1} \in Q$ which is a contradiction. Therefore,  ${\alpha_i^j}^,$s are different. From $a_{\alpha_1^j}, a_{\alpha_2^j},\ldots,a_{\alpha_{n_j}^j}  \in I_j$ for all $1 \leq j \leq n$ it follows that $a_1 \cdots a_n=a_{\alpha_1^1}  a_{\alpha_2^1}  \cdots  a_{\alpha_{n_1}^1} \cdots a_{\alpha_1^k}   a_{\alpha_2^k}   \cdots   a_{\alpha_{n_k}^k} \in \cap_{j=1}^k I_j$. Consequently, $\cap_{j=1}^k I_j$ is an $(n,Q)$-ideal of $A$.
\end{proof}

Let     $I$ be a proper ideal of $A$.  Then, we  use $Z_I(A)$ to denote  the set $\{x \in A \mid xs \in I \  \text{for some} \ s \in A \backslash I\}$.
 \begin{theorem} \label{5}
Let $I$ and $Q$ be ideals of $A$ and $S$ be a multiplicative subset of $A$. 
\begin{enumerate}
\item If $I$ is an $(n,Q)$-ideal of $A$ such that  $Q \cap S=\varnothing$, then $S^{-1}I$ is an $(n,S^{-1}Q)$-ideal of $S^{-1}A$.
\item If $S^{-1}I$ is  an $(n,S^{-1}Q)$-ideal of $S^{-1}A$ and $S \cap Z_j=\varnothing$ for $j \in\{I,Q\}$, then $I$ is an $(n,Q)$-ideal of $A$.
\end{enumerate} 
\end{theorem}
\begin{proof}
(1) Let $I$ be an $(n,Q)$-ideal of $A$. Assume that $  \frac{a
_1}{s_1}\cdots\frac{a_{n+1}}{s_{n+1}} \in S^{-1}I$ for  $ \frac{a
_1}{s_1},\ldots,\frac{a_{n+1}}{s_{n+1}} \in S^{-1}A$ such that $  \frac{a_1}{s_1}\cdots\frac{a_n}{s_n} \notin S^{-1}I$. Hence, we get $ sa_1\cdots a_{n+1}=a_1\cdots a_n(sa_{n+1}) \in I$ for some $s \in S$. Since $I$ is an $(n,Q)$-ideal  of $A$ and $a_1\cdots a_n \notin I$, we get $a_1\cdots \widehat{a_t} \cdots a_n (sa_{n+1})\in Q$ for some $1 \leq t \leq n$ which means  $ \frac{a
_1}{s_1}\cdots \widehat{\frac{a_t}{s_t}} \cdots \frac{a_{n+1}}{s_{n+1}} \in S^{-1}Q$. Thus,  $S^{-1}I$ is an $(n,S^{-1}Q)$-ideal of $S^{-1}A$.

(2) Let  $S^{-1}I$ be an $(n,S^{-1}Q)$-ideal of $S^{-1}A$. Suppose that $  a_1 \cdots a_{n+1} \in I$ for $a_1, \ldots, a_{n+1} \in A$ and $  a_1 \cdots a_n \notin I$ . This means that $\frac{a_1}{1} \cdots \frac{a_{n+1}}{1} \in S^{-1}I$. If $\frac{a_1}{1} \cdots \frac{a_n}{1} \in S^{-1}I$, then we have $sa_1\cdots a_n \in I$ for some $s \in S$. Since  $S \cap Z_I(A)=\varnothing$, we have  $a_1\cdots a_n \in I$, a contradiction. Since $S^{-1}I$ is an $(n,S^{-1}Q)$-ideal of $S^{-1}A$ and $\frac{a_1}{1} \cdots \frac{a_n}{1} \notin S^{-1}I$, we have $\frac{a_1}{1} \cdots \widehat{\frac{a_t}{1}} \cdots  \frac{a_{n+1}}{1} \in S^{-1}Q$ for some $1 \leq t \leq n$. It follows that $sa_1\cdots \widehat{a_t} \cdots a_{n+1} \in Q$. Since $S \cap Z_Q=\varnothing$, we get $a_1\cdots \widehat{a_t} \cdots a_{n+1} \in Q$. Consequently, $I$ is an $(n,Q)$-ideal of $A$.
\end{proof}
\begin{theorem} \label{6}
Let $A= A_1 \times \cdots \times A_r$ be a decomposable ring and let \[ I = I_1 \times \cdots \times I_{\beta_1-1} \times A_{\beta_1} \times I_{\beta_1+1} \times \cdots \times I_{\beta_k-1} \times A_{\beta_k} \times I_{\beta_k+1} \times \cdots \times I_r \] and \[ Q = Q_1 \times \cdots \times Q_{\beta_1-1} \times A_{\beta_1} \times Q_{\beta_1+1} \times \cdots \times Q_{\beta_k-1} \times A_{\beta_k} \times Q_{\beta_k+1} \times \cdots \times Q_r \]be   ideals of $A$ with    $\{\beta_1, \dots, \beta_k\} \subset \{1, \dots, r\}$.  The following conditions are equivalent:
 \begin{enumerate} 
 \item[(1)] $I$ is an $(n,Q)$-ideal of $A$; 
 \item[(2)] $I'  = I_1 \times \cdots \times I_{\beta_1-1} \times I_{\beta_1+1} \times \cdots \times I_{\beta_k-1} \times I_{\beta_k+1} \times \cdots \times I_r$ is an $(n,Q')$-ideal of $A' =A_1 \times \cdots \times A_{\beta_1-1} \times A_{\beta_1+1} \times \cdots \times A_{\beta_k-1} \times A_{\beta_k+1} \times \cdots \times A_r$ where $Q'=Q_1 \times \cdots \times Q_{\beta_1-1} \times Q_{\beta_1+1} \times \cdots \times Q_{\beta_k-1} \times Q_{\beta_k+1} \times \cdots \times I_r$.
  \end{enumerate}
\end{theorem}
\begin{proof}
(1) $\Longrightarrow$ (2) Let $I$ be an $(n,Q)$-ideal of $A$ and \\

$\hspace{1cm} (a^{(1)}_1,\ldots,a^{(1)}_{\beta_1-1}, a^{(1)}_{\beta_1+1},\cdots, a^{(1)}_{\beta_k-1}, a^{(1)}_{\beta_k+1}, \dots,a_r^{(1)})\cdots$\\

$\hspace{2cm}(a^{(n+1)}_1,\ldots,a^{(n+1)}_{\beta_1-1}, a^{(n+1)}_{\beta_1+1},\ldots, a^{n+1}_{\beta_k-1}, a^{(n+1)}_{\beta_k+1},\ldots,a_r^{(n+1)}) \in I'$

such that  $a_i^{(k),}s$  are in $A_i$. This implies that\\

$\hspace{1cm} (a^{(1)}_1,\ldots,a^{(1)}_{\beta_1-1}, 1, a^{(1)}_{\beta_1+1},\cdots, a^{(1)}_{\beta_k-1}, 1,a^{(1)}_{\beta_k+1}, \dots,a_r^{(1)})\cdots$\\

$\hspace{2cm}(a^{(n+1)}_1,\ldots,a^{(n+1)}_{\beta_1-1},1, a^{(n+1)}_{\beta_1+1},\ldots, a^{(n+1)}_{\beta_k-1},1, a^{(n+1)}_{\beta_k+1},\ldots,a_r^{(n+1)}) \in I.$\\

Since $I$ is an $(n,Q)$-ideal of $A$, we conclude that\\

$\hspace{1cm} (a^{(1)}_1,\ldots,a^{(1)}_{\beta_1-1}, 1, a^{(1)}_{\beta_1+1},\cdots, a^{(1)}_{\beta_k-1}, 1,a^{(1)}_{\beta_k+1}, \dots,a_r^{(1)})\cdots$\\

$\hspace{2cm}(a^{(n)}_1,\ldots,a^{(n)}_{\beta_1-1},1, a^{(n)}_{\beta_1+1},\ldots, a^{(n)}_{\beta_k-1},1, a^{(n)}_{\beta_k+1},\ldots,a_r^{(n)}) \in I$\\

or\\

$\hspace{0.75cm} (a^{(1)}_1,\ldots,a^{(1)}_{\beta_1-1}, 1, a^{(1)}_{\beta_1+1},\cdots, a^{(1)}_{\beta_k-1}, 1,a^{(1)}_{\beta_k+1}, \dots,a_r^{(1)})\cdots$\\

$\hspace{1.75cm}(a^{(t-1)}_1,\ldots,a^{(t-1)}_{\beta_1-1},1, a^{(t-1)}_{\beta_1+1},\ldots, a^{(t-1)}_{\beta_k-1},1, a^{(t-1)}_{\beta_k+1},\ldots,a_r^{(t-1)}) $\\

$\hspace{2cm}(a^{(t+1)}_1,\ldots,a^{(t+1)}_{\beta_1-1},1, a^{(t+1)}_{\beta_1+1},\ldots, a^{(t+1)}_{\beta_k-1},1, a^{(t+1)}_{\beta_k+1},\ldots,a_r^{(t+1)}) \cdots$\\

$\hspace{2.25cm}(a^{(n+1)}_1,\ldots,a^{(n+1)}_{\beta_1-1},1, a^{(n+1)}_{\beta_1+1},\ldots, a^{(n+1)}_{\beta_k-1},1, a^{(n+1)}_{\beta_k+1},\ldots,a_r^{(n+1)}) \in Q$\\

for some $1 \leq t \leq n$. Then, we obtain \\

$\hspace{1cm} (a^{(1)}_1,\ldots,a^{(1)}_{\beta_1-1},  a^{(1)}_{\beta_1+1},\cdots, a^{(1)}_{\beta_k-1}, a^{(1)}_{\beta_k+1}, \dots,a_r^{(1)})\cdots$\\

$\hspace{2cm}(a^{(n)}_1,\ldots,a^{(n)}_{\beta_1-1},  a^{(n)}_{\beta_1+1},\ldots, a^{(n)}_{\beta_k-1},  a^{(n)}_{\beta_k+1},\ldots,a_r^{(n)}) \in I'$\\

or\\

$\hspace{0.75cm} (a^{(1)}_1,\ldots,a^{(1)}_{\beta_1-1},   a^{(1)}_{\beta_1+1},\cdots, a^{(1)}_{\beta_k-1},  a^{(1)}_{\beta_k+1}, \dots,a_r^{(1)})\cdots$\\

$\hspace{1.75cm}(a^{(t-1)}_1,\ldots,a^{(t-1)}_{\beta_1-1},  a^{(t-1)}_{\beta_1+1},\ldots, a^{(t-1)}_{\beta_k-1},  a^{(t-1)}_{\beta_k+1},\ldots,a_r^{(t-1)}) $\\

$\hspace{2cm}(a^{(t+1)}_1,\ldots,a^{(t+1)}_{\beta_1-1},  a^{(t+1)}_{\beta_1+1},\ldots, a^{(t+1)}_{\beta_k-1},  a^{(t+1)}_{\beta_k+1},\ldots,a_r^{(t+1)}) \cdots$\\

$\hspace{2.25cm}(a^{(n+1)}_1,\ldots,a^{(n+1)}_{\beta_1-1},  a^{(n+1)}_{\beta_1+1},\ldots, a^{(n+1)}_{\beta_k-1},  a^{(n+1)}_{\beta_k+1},\ldots,a_r^{(n+1)}) \in Q'$\\
Thus, $I'$ is an $(n,Q')$-ideal of $A'$.

(2) $\Longleftarrow$ (1)   Suppose that $I'  = I_1 \times \cdots \times I_{\beta_1-1} \times I_{\beta_1+1} \times \cdots \times I_{\beta_k-1} \times I_{\beta_k+1} \times \cdots \times I_r$ is an $(n,Q')$-ideal of $A' =A_1 \times \cdots \times A_{\beta_1-1} \times A_{\beta_1+1} \times \cdots \times A_{\beta_k-1} \times A_{\beta_k+1} \times \cdots \times A_r$ where $Q'=Q_1 \times \cdots \times Q_{\beta_1-1} \times Q_{\beta_1+1} \times \cdots \times Q_{\beta_k-1} \times Q_{\beta_k+1} \times \cdots \times I_r$. Using an argument similar to the previous part, we  can   see  that $I$ is an $(n,Q)$-ideal of $A$.
\end{proof}

\begin{theorem}\label{7} 
Let $A = A_1 \times \cdots \times A_r$ be a ring with identity and let  $I=I_1  \times \cdots \times I_r$ be  an ideal of $A$ in which  $I_i \subseteq Q_i$ are ideals of $A_i$ for each $1 \le i \le r$. The following statements are equivalent:
 \begin{enumerate}
  \item[(1)] $I$ is an $(n,Q)$-ideal of $A$ where $Q =Q_1 \times \cdots \times Q_{r}$; 
 \item[(2)] $I_1  \times \cdots \times I_{r-1}$ is an $(n,Q')$-ideal of $A_1 \times   \cdots \times A_{r-1}$ where $Q'=Q_1 \times \cdots \times Q_{r-1}$ and $Q_r = A_r$   or   $I_i$ is a $Q_i$-ideal of $A_i$ for every $1 \le i \le r$. 
 \end{enumerate}
 \end{theorem}
 \begin{proof}
(1) $\Longrightarrow$ (2) Let $I$ be an $(n,Q)$-ideal of $A$ where $Q =Q_1 \times \cdots \times Q_{r}$. Then, $I_1  \times \cdots \times I_{r-1}$ is an $(n,Q')$-ideal of $A_1 \times   \cdots \times A_{r-1}$ where $Q'=Q_1 \times \cdots \times Q_{r-1}$ by  Theorem \ref{6}.  Assume that $Q_r \neq A_r$, $a \in I_r$ and $xy \in I_j $ for $x,y \in A_j$ and $1 \leq j \leq r-1$ such that $x \notin Q_j$. Therefore, we have \\
 
 $(1,\ldots,1,\underbrace{x}_{j-th},1,\ldots,1)(1,\ldots,1,\underbrace{y}_{j-th},1,\ldots,1)(1,0,1,\ldots,1)(1,1,0,1,\ldots,1) \cdots $
 
 $(1,\ldots,1,0,\underbrace{1}_{j-th},1,\ldots,1)(1,\ldots,1,\underbrace{1}_{j-th},0,1,\ldots,1)\cdots(1,\ldots,1,0)(1,\ldots,1,a)$
 
 $\hspace{2cm}=(0,\ldots,0, \underbrace{xy}_{j-th},0\ldots,0,a) \in  I.$\\
 
  Since $I$ is an $(n,Q)$-ideal of $A$ and  \\
  
   $\hspace{1cm}(1,\ldots,1,\underbrace{x}_{j-th},1,\ldots,1)(1,0,1,\ldots,1)(1,1,0,1,\ldots,1) \cdots $
 
 $\hspace{1.5cm}(1,\ldots,1,0,\underbrace{1}_{j-th},1,\ldots,1)(1,\ldots,1,\underbrace{1}_{j-th},0,1,\ldots,1)\cdots$
 
 $\hspace{2cm}(1,\ldots,1,0)(1,\ldots,1,a) $
 
 $\hspace{2.5cm}=(0,\ldots,0, \underbrace{x}_{j-th},0\ldots,0,a) \notin  Q$,\\
 
   we get\\
 
 $\hspace{1cm}(1,\ldots,1,\underbrace{y}_{j-th},1,\ldots,1)(1,0,1,\ldots,1)(1,1,0,1,\ldots,1) \cdots $
 
 $\hspace{1.5cm}(1,\ldots,1,0,\underbrace{1}_{j-th},1,\ldots,1)(1,\ldots,1,\underbrace{1}_{j-th},0,1,\ldots,1)\cdots$
 
 $\hspace{2cm}(1,\ldots,1,0)(1,\ldots,1,a) $
 
 $\hspace{2.5cm}=(0,\ldots,0, \underbrace{y}_{j-th},0\ldots,0,a) \in I$,\\
 which means $y \in I_j$. Hence, $I_j$ is a $Q_j$-ideal of $A_i$ for all $1 \leq j \leq r-1$. By a similar argument, we conclude that $I_r$ is a $Q_r$-ideal of $A_r$.
 
 (2) $\Longrightarrow$ (1) Let $I_1  \times \cdots \times I_{r-1}$ be an $(n,Q')$-ideal of $A_1 \times   \cdots \times A_{r-1}$ where $Q'=Q_1 \times \cdots \times Q_{r-1}$ and $Q_r = A_r$. By Theorem \ref{6}, $I$ is an $(n,Q)$-ideal of $A$ where $Q =Q_1 \times \cdots \times Q_{r}$. Assume that $I_i$ is a $Q_i$-ideal of $A_i$ for every $1 \le i \le r$ and 
 $(a_1^{(1)},\ldots,a_r^{(1)})\cdots( a_1^{(n+1)},\ldots,a_r^{(n+1)}) \in I$ such that $a^{(k),}_i$s are in $A_i$. Then, $a_i^{(1)}\cdots a_i^{(n+1)} \in I_i$ for any $1 \leq i \leq r$. Then, there exists $1 \leq k \leq n+1$ such that $a_i^{(k)} \in I_i$ or $a_i^{(k)} \in Q_i$. Thus, $I$ is an $(n,Q)$-ideal of $A$ where $Q =Q_1 \times \cdots \times Q_{r}$.
 \end{proof}
\section{$\phi$-$(n,Q)$-ideals}
\begin{definition}
Let $I$ and $ Q$  be proper ideals of $A$  and $\phi: \mathcal{I}(A) \to \mathcal{I}(A) \cup \{\varnothing\}$  be a function. Then, $I$ is called a $\phi$-$(n,Q)$-ideal of $A$ if whenever $a_1\cdots a_{n+1} \in I \backslash \phi(I)$ for $a_1,\ldots,a_{n+1} \in A$, then $a_1\cdots  a_n \in I$ or there exists $1 \leq t \leq n$ such that $a_1\cdots \widehat{a_t} \cdots a_{n+1} \in Q.$
\end{definition}
 \begin{proposition}  
Let $I $   be a proper ideal of $A$ and $\phi: \mathcal{I}(A) \to \mathcal{I}(A) \cup \{\varnothing\}$ be a function.   If $I$ is a $\phi$-$(n,Q)$-ideal of $A$ for    some  proper ideal $Q$ of $A$, then $I \backslash \phi(I)  \subseteq Q$.
\end{proposition}
\begin{proof}
Assume  that $I$ is a    $\phi$-$(n,Q)$-ideal where $Q$ is a proper ideal of $A$. Let  $ I \backslash \phi(I) \nsubseteq Q$. Then,    there exists $a \in I \backslash \phi(I)$ such that $a \notin Q$. Since $I$ is a  $\phi$-$(n,Q)$-ideal of $A$ and $1 \cdots 1\cdot  a\in I$ , we obtain $1  \in I$, a contradiction. Therefore, $I \subseteq \phi(I)  \subseteq Q$.
\end{proof}
\begin{theorem}
Let $I \subseteq Q$ be ideals of $A$ and $\phi: \mathcal{I}(A) \to \mathcal{I}(A) \cup \{\varnothing\}$ be a function. Then, the following statements are equivalent:
\begin{enumerate}
\item $I$ is a $\phi$-$(n,Q)$-ideal of $A$;
\item If $a_1 \cdots a_n \notin Q$ for $a_1,\ldots, a_n \in A$, then \[(I: a_1 \cdots a_n) \subseteq (I:a_1 \cdots a_{n-1}) \cup [\cup_{t=1}^{n-1}(Q:a_1 \cdots \widehat{a_t} \cdots a_n)] \cup (\phi(I):a_1 \cdots a_n).\]
\end{enumerate}
\end{theorem}
\begin{proof}
(1) $\Longrightarrow$ (2) Let $I$ be a $\phi$-$(n,Q)$-ideal of $A$ and  $a_1 \cdots a_n \notin Q$ for $a_1,\ldots, a_n \in A$. Assume that $x \in (I: a_1 \cdots a_n)$. So, $xa_1 \cdots a_n \in I$. Suppose that $xa_1 \cdots a_n \in   \phi(I).$ This implies that $x \in (\phi(I):a_1 \cdots a_n)$. Let $xa_1 \cdots a_n  \notin \phi(I).$ Since $I$ is a $\phi$-$(n,Q)$-ideal of $A$ and $a_1 \cdots a_n \notin Q$, we obtain $xa_1 \cdots a_{n-1} \in I$ or $xa_1 \cdots \widehat{a_t} \cdots a_{n-1} \in Q$ for some $1 \leq t \leq n-1$. This means that   $x \in (I:a_1 \cdots a_{n-1})$ or $x \in (Q:a_1 \cdots \widehat{a_t} \cdots a_n)$. Thus, \[(I: a_1 \cdots a_n) \subseteq (I:a_1 \cdots a_{n-1}) \cup [\cup_{t=1}^{n-1}(Q:a_1 \cdots \widehat{a_t} \cdots a_n)] \cup (\phi(I):a_1 \cdots a_n).\]

(2) $\Longrightarrow$ (1) Assume that  $a_1,\ldots,a_{n+1} \in A$ such that $a_1\cdots a_{n+1} \in I \backslash \phi(I)$ but $a_1\cdots   a_n \notin I$. So, we have $a_1 \in  (I:a_2 \cdots a_{n+1})$. Let $a_2 \cdots a_{n+1} \in Q$. Then, $Q$ is a $\phi$-$(n,Q)$-ideal of $A$. Now, let us assume that  $a_2 \cdots a_{n+1} \notin Q$. By the hypothesis, we obtain $a_1 \in \cup_{t=2}^{n}(Q:a_2 \cdots \widehat{a_t} \cdots a_{n+1})$ as $a_1 \cdots  a_n \notin I$ and $a_1 \cdots a_{n+1} \notin \phi(I)$ which implies $a_1\cdots \widehat{a_t} \cdots a_{n+1} \in Q $ for some $2 \leq t \leq n$. Hence, $Q$ is a $\phi$-$(n,Q)$-ideal of $A$.
\end{proof}
\begin{proposition}
Let   $\phi: \mathcal{I}(A) \to \mathcal{I}(A) \cup \{\varnothing\}$ be a function.
\begin{enumerate}
\item If $I$ is a $\phi$-$(n,Q)$-ideal   of $A$ for each ideal $Q$   properly containing $I$, then $I$ is a $\phi$-$n$-absorbing primary ideal.
\item  If $I$ is a $\phi$-$(n,Q)$-ideal of $A$ for each ideal $Q$   properly containing $\sqrt{I}$ and $\sqrt{\phi(I)} =\phi(\sqrt{I})$, then $\sqrt{I}$ is an $\phi$-$n$-absorbing  ideal.
\item The following conditions are equivalent:
\begin{enumerate}
\item[(i)] $I$ is a $\phi$-$n$-absorbing primary ideal.
\item[(ii)] $I$ is a  $\phi$-$(n,\sqrt{I})$-ideal.
\item[(iii)] $I$ is a  $\phi$-$(n,Q)$-ideal of $A$ for each ideal $Q$ of $A$ containing $\sqrt{I}$.
\end{enumerate}
\end{enumerate}
\end{proposition}
\begin{proof}
(1) The proof is similar to that in the proof of Theorem \ref{2.1}(1). 

(2) Let $I$ be  an $(n,Q)$-ideal of $A$ for each ideal $Q$ of $A$ properly containing $\sqrt{I}$. Suppose  that $a_1 \cdots a_{n+1} \in \sqrt{I} \backslash \phi(\sqrt{I})$ and $a_1 \cdots a_n \notin \sqrt{I}$. This means  that $(a_1 \cdots a_{n+1})^k=a_1^k \cdots a_{n+1}^k\in I$ for some $k \in \mathbb{N}$. If $a_1^k \cdots a_{n+1}^k \in \phi(I)$, then  $a_1  \cdots a_{n+1} \in \sqrt{\phi(I)} =\phi(\sqrt{I})$, a contradiction. Then, $a_1^k \cdots a_{n+1}^k \in I \backslash \phi(I)$. Now, in a similar manner to the proof of Theorem \ref{2.1}(2),  one can complete the proof.

(3) (i) $\Longrightarrow$ (ii) and (ii) $\Longrightarrow$ (iii) are clear.

(iii) $\Longrightarrow$ (i)  Following a similar argument as in the proof of Theorem \ref{2.1}(3)((iii) $\Longrightarrow$ (i)), the claim is proved.
\end{proof}
Let   $\phi: \mathcal{I}(A) \to \mathcal{I}(A) \cup \{\varnothing\}$ be a function, $I$  an $\phi$-$(n,Q)$-ideal of $A$ and $a_1,\cdots, a_{n+1} \in A$.  Then, $(a_1, \cdots, a_{n+1})$ is called a $\phi$-$(n+1,Q)$-tuple of $I$ if $a_1  \cdots  a_{n+1} \in \phi(I)$, $a_1  \cdots  a_n \notin I$ and $a_1  \cdots  \widehat{a_t} \cdots a_{n+1} \notin Q$ for all $1 \leq t \leq n$. 

Let     $A$ is  a commutative ring with identity $1$ and  $M$ be an $A$-module. Recall
from \cite{Anderson} that   $A(+)M=\{(a,x) \mid a \in A, x \in M\}$ with  the  addition  $ (a, x) + (b, y) = (a + b, x +y) $ and the multiplication   $(a, x)(b, y) = (ab, bx + ay)$  is a  ring with identity $(1, 0)$.  This   ring is called the idealization of  $M$. Suppose  that  $I$ is an ideal of $A$ and $F$ a submodule of $M$. Theorem 3.3 in \cite{Anderson} indicates  that  $IM \subseteq F$  if and only if $I(+)F$ is an ideal of $A(+)M$.

 Here, we determine the relation between $\phi$-$(n,Q)$-ideals of $A$ and $\psi$-$(n,Q(+)M)$-ideals of $A(+)M$. 
 \begin{theorem} \label{idealization}
 Let $I$ and $Q$ be proper  ideals of $A$ and $M$ be  $A$-module. Assume that  $\phi: \mathcal{I}(A) \to \mathcal{I}(A) \cup \{\varnothing\}$ and $\psi: \mathcal{I}(A(+)M) \to \mathcal{I}(A(+)M) \cup \{\varnothing\}$ are two  functions with  $\phi(I)M\subseteq F$ for some submodule $F$ of $M$ and $\psi(I (+) M)=\phi(I) (+) F$. Then, the following statements are equivalent:
\begin{enumerate}
\item   $I$ is a $\phi$-$(n,Q)$-ideal of $A$, and if  $(a_1,\ldots,a_{n+1})$  is a $\phi$-$(n+1,Q)$-tuple of $I$, then the second component of $(a_1, x_1)\cdots (a_{n+1},x_{n+1})$ is in $F$ for every $x_1,\ldots x_{n+1} \in M$;
\item $I(+)M$ is a $\psi$-$(n,Q(+)M)$-ideal of  $A(+)M$.
\end{enumerate}
 \end{theorem}
 \begin{proof}
 (1) $\Longrightarrow$ (2) Assume that $(a_1, x_1)\cdots (a_{n+1},x_{n+1}) \in  I(+)M  \backslash \phi(I) (+) F$ for $a_1,\ldots, a_{n+1} \in A$ and $x_1,\ldots,x_{n+1} \in M$. Let $a_1 \cdots a_{n+1} \in \phi(I)$. Since $a_1 \cdots a_{n+1} \in I$  and $(a_1, x_1)\cdots (a_{n+1},x_{n+1}) \notin  \phi(I) (+) F$,  we conclude that the second component of $(a_1, x_1)\cdots (a_{n+1},x_{n+1})$ is not in $F$. Therefore, $(a_1,\ldots,a_{n+1})$   is not a $\phi$-$(n+1,Q)$-tuple of $I$. This means $a_1 \cdots  a_n \in I$ or $a_1 \cdots  \widehat{a_t} \cdots a_n \in Q$ for some $1 \leq t \leq n$. Now, let $a_1 \cdots a_{n+1} \notin \phi(I)$. Since $I$ is a $\phi$-$(n,Q)$-ideal of $A$ and $a_1 \cdots a_{n+1} \in I$, we get $a_1 \cdots  a_n \in I$ or $a_1 \cdots  \widehat{a_t} \cdots a_n \in Q$ for some $1 \leq t \leq n$, again. This implies that $(a_1, x_1)\cdots (a_n,x_n) \in  I(+)M $ or $(a_1, x_1)\cdots \widehat{(a_t, x_t)} \cdots (a_{n+1},x_{n+1}) \in  Q(+)M$ for some $1 \leq t \leq n$. Thus, $I(+)M$ is a $\psi$-$(n,Q(+)M)$-ideal of  $A(+)M$.
 
 (2) $\Longrightarrow$ (1) Let  $I(+)M$ be  a $\psi$-$(n,Q(+)M)$-ideal of  $A(+)M$. Assume that $a_1 \cdots a_{n+1} \in I \backslash \phi(I)$ for $a_1, \ldots, a_{n+1} \in A $. Since $\psi(I (+) M)=\phi(I) (+) F$, we have  $(a_1, 0)\cdots (a_{n+1},0) \in  I(+)M  \backslash \psi( I  (+) M)$. Then, we get $(a_1, 0)\cdots (a_n,0) \in  I(+)M $ or $(a_1,0)\cdots \widehat{(a_t, 0)} \cdots (a_{n+1},0) \in  Q(+)M$ for some $1 \leq t \leq n$. This means that $a_1 \cdots  a_n \in I$ or $a_1 \cdots  \widehat{a_t} \cdots a_n \in Q$ for some $1 \leq t \leq n$, as needed. Now, assume that  $a_1 \cdots  a_{n+1} \in \phi(I)$  but $a_1 \cdots  a_n \notin I$ and $a_1 \cdots  \widehat{a_t} \cdots a_n \notin Q$ for every  $1 \leq t \leq n$. Let us assume that the second component of $(a_1, x_1)\cdots (a_{n+1},x_{n+1})$ is not in $F$ for   $x_1,\ldots x_{n+1} \in M$. Hence, we have $(a_1, x_1)\cdots (a_{n+1},x_{n+1}) \in  I(+)M  \backslash \phi(I) (+) F$. Therefore, we conclude that  $(a_1, x_1)\cdots (a_n,x_n) \in  I(+)M $, i.e., $a_1 \cdots  a_n \in I$ or $(a_1, x_1)\cdots \widehat{(a_t, x_t)} \cdots (a_{n+1},x_{n+1}) \in  Q(+)M$ for some $1 \leq t \leq n$, i.e., $a_1 \cdots  \widehat{a_t} \cdots a_n \in Q$, a contradiction.
 \end{proof}
 \begin{theorem}
 Let $I$ and $Q$ be proper  ideals of $A$ and $M$ be  $A$-module. Then,   $I$ is an  $(n,Q)$-ideal of $A$ if and only if $I (+) M$ is an $(n,Q(+)M)$-ideal of $A (+) M$. 
 \end{theorem}
 \begin{proof}
 Take  $F=M$, $\phi=\varnothing$ and  $\psi=\varnothing$  in Theorem \ref{idealization}.
 \end{proof}


\end{document}